\documentclass[11pt, a4paper]{amsart}
\usepackage{amsmath,amssymb,amsfonts, float}
\usepackage[all]{xy}
\usepackage{caption}
\usepackage{enumerate}
\usepackage{mathpazo}
\usepackage{a4wide}
\usepackage{diagbox}
\usepackage{subcaption}

\usepackage{bm}
\usepackage{graphicx}
\usepackage[breaklinks=true]{hyperref}

\usepackage{xcolor}
\usepackage{multicol}
\usepackage{tabularx}

\usepackage{hyperref}
\usepackage[capitalize]{cleveref}

\usepackage[normalem]{ulem}

\newtheorem{thm}{Theorem}[section] 
\newtheorem*{thm*}{Theorem} 
\newtheorem{prop}[thm]{Proposition}

\newtheorem{cor}[thm]{Corollary}

\theoremstyle{definition}
\newtheorem{definition}[thm]{Definition}

\newtheorem{conj}[thm]{Conjecture}

\newtheorem{question}[thm]{Question}
\newtheorem{rem}[thm]{Remark}

\DeclareMathOperator{\C}{\mathbb{C}}
\DeclareMathOperator{\Z}{\mathbb{Z}}

\DeclareMathOperator{\F}{\mathbb{F}}

\DeclareMathOperator{\Hom}{{\rm Hom}}

\newcommand{\Ann}{{\rm Ann}}

    \DeclareFontFamily{U}{wncy}{}
    \DeclareFontShape{U}{wncy}{m}{n}{<->wncyr10}{}
    \DeclareSymbolFont{mcy}{U}{wncy}{m}{n}
    \DeclareMathSymbol{\Sha}{\mathord}{mcy}{"58}

\numberwithin{equation}{section}

\DeclareSymbolFont{bbold}{U}{bbold}{m}{n}
\DeclareSymbolFontAlphabet{\mathbbold}{bbold}

\usepackage{hyperref}

\newcommand{\Rad}{{\rm Rad}}
\newcommand{\s}{\text{ss}}

\newcommand{\diag}{\text{diag}}

\title{Supercharacter theory and applications to \\ Ramanujan sums over a finite Frobenius ring}
 \author{Tung T. Nguyen, Nguyen Duy T\^{a}n, Enrique Trevi\~no}

 \address{Department of Mathematics, Elmhurst University, Elmhurst,  Illinois, USA}
 \email{tung.nguyen@elmhurst.edu}
 
  \address{
Faculty Mathematics and Informatics, Hanoi University of Science and Technology, 1 Dai Co Viet Road, Hanoi, Vietnam } 
\email{tan.nguyenduy@hust.edu.vn}

 \address{Department of Mathematics and Computer Science, Lake Forest College, Lake Forest, Illinois, USA}
 \email{trevino@lakeforest.edu}
 
\thanks{TTN is partially supported by an AMS-Simons Travel Grant.  NDT is partially supported by the Vietnam National
Foundation for Science and Technology Development (NAFOSTED) under grant number 101.04-2023.21}
\begin{document}
\maketitle
\begin{abstract}
The theory of  Ramanujan sums has been playing a fundamental role in several subfields of mathematics. They appear in the theory of special values of zeta functions, spectral graph theory, representation theory, and analytic number theory. Recent work has shown that classical Ramanujan sums can also be interpreted as a super-Fourier transform via the theory of supercharacters for the ring $\Z/n$. In this article, building upon our recent work on supercharacters over an arbitrary finite Frobenius ring, we explore additional arithmetical properties of Ramanujan sums. Our approach provides a unified framework that generalizes various results in the literature regarding these sums. As a by-product, we also describe a new criterion for determining when a finite commutative ring is Frobenius, which could be of independent interest to the algebra community.

\end{abstract}

\section{Introduction}

Let $n$ be a positive integer and $\zeta_n$ be a fixed primitive $n$-th root of unity. The sum 
\begin{equation}  \label{eq:classical_ramanujan}
c_n(m) = \sum_{\substack{1 \leq j \leq n \\ \gcd(j,n)=1}} \zeta_n^{mj},
\end{equation}
is known in the literature as a Ramanujan sum. Although Dirichlet and Dedekind had considered this sum in the mid-19th century, it was Ramanujan who first realized its importance in the early 20th century, using it to investigate several problems in number theory (\cite[Page 159]{hardy1999ramanujan}). For instance, Ramanujan used these sums to derive new expressions for arithmetical functions such as the divisor function. Ramanujan sums have since found applications in various subfields of mathematics, including representation theory, analytic number theory, sieve theory, graph theory, and physics. In graph theory, for example, these sums have been used to study the spectra of certain classes of graphs. We refer the reader to \cite[Section 1.1]{fowler2014ramanujan} for a more extensive discussion on the history and applications of Ramanujan sums.

Our own interest in Ramanujan sums stems from their recurring appearance in our research.  These sums, together with Gauss sums, first appear in our calculations of the spectrum of the generalized Paley graph associated with a quadratic character (see \cite{paleygraph}). They reappear in our investigation of Fekete polynomials associated with principal Dirichlet characters. More precisely, we recall that the $n$-th Fekete polynomial is defined as 
\begin{equation}
        F_n(x)= \sum_{\substack{1 \leq a \leq n\ \gcd(a,n)=1}} x^a.
\end{equation}
By definition, $c_n(m)$ is precisely the value of $F_n$ at an $n$-th root of unity; namely $F_n(\zeta_n^m).$ Using the explicit formula for $c_n(m)$, we can show that, if $n$ is squarefree and $d$ is a divisor of $n$, the cyclotomic polynomial $\Phi_d$ is not a factor of $F_n$ (see \cite[Corollary 2.7]{chidambaram2023fekete}). Later on, while working on prime Cayley graphs (see \cite{chudnovsky2024prime}), we found the work of Klotz-Sander and So on unitary graphs and gcd-graphs where Ramanujan sums play central roles (see \cite{klotz2007some, so2006integral}). For example, in \cite{so2006integral}, So uses Ramanujan sums to classify all integral circulant graphs, which effectively resolves a conjecture of Klotz and Sander on integral circulant graphs stated in \cite{klotz2007some}. While reading more work on gcd-graphs, we soon realized that the theory of gcd-graphs can be generalized to an arbitrary finite commutative ring. Furthermore, when the underlying ring is a Frobenius ring, we can even develop a general theory of Ramanujan sums and utilize them to calculate explicitly the spectra of the associated gcd-graphs. This circle of ideas has led us to various works in this research direction (see \cite{minavc2024gcd, nguyen2025gcd, nguyen2024integral, nguyen2025perfect}). In this article, building upon recent advances on supercharacter theory and its applications to classical Ramanujan sums (see \cite{diaconis2008supercharacters, fowler2014ramanujan, supercharacters_nguyen}), we study some further arithmetical properties of generalized Ramanujan sums. Along the way, we discuss some connections with spectral graph theory. Additionally, we investigate the determinant of a matrix associated with Ramanujan sums. Using the theory of supercharacters for Frobenius rings developed in \cite{supercharacters_nguyen}, we determine the precise value of this determinant (up to a sign). 

We remark that the generalized theory of Ramanujan sums over finite Frobenius rings was pioneered by Lamprecht in his 1953 work \cite{lamprecht1953allgemeine}. However, his contributions in this area have not received widespread recognition as it should in the mathematical community (see  \cite{honold2001characterization} for some further valuable historical context on Lamprecht's work on Frobenius rings). Given the significance of Lamprecht's early insights, we feel it is important to acknowledge his foundational role in this line of research. 

\subsection{Outline}
The article is structured as follows. In \cref{subsec:supercharacter_group}, we review the theory of supercharacters over a finite abelian group. The key ideas in this section have been previously discussed in \cite{fowler2014ramanujan}. Our main contribution here is the introduction of certain related sums, which appear naturally in the spectral description of certain Cayley graphs. In \cref{subsec:supercharacter_ring}, we apply the results from the previous section to the case where $R$ is a finite commutative Frobenius ring.  More precisely, we explain the existence of a natural supercharacter theory on $R$ and show how this theory is closely related to the theory of Ramanujan sums developed in \cite{nguyen2025gcd}. Using the general results in \cref{subsec:supercharacter_group}, we derive various orthogonality relations for these Ramanujan sums. We also explain how our results recover some well-known formulas in the literature. Furthermore, we give an explicit formula for the $k$-th moment of these Ramanujan sums for each $k \geq 1$. In \cref{sec:determinant}, we explicitly describe the supercharacter table for the associated supercharacter theory outlined in \cref{subsec:supercharacter_ring}. Additionally, we show that  the determinant of this supercharacter table can determine whether a finite commutative ring is Frobenius or not. Our main theorem provides a unified proof for various special cases studied in \cite{schlage2021determinant} (for the case where $R$ is $\Z/n$) and \cite{minavc2024gcd} (for the case where $R$ is a quotient of $\F_q[x]$). Finally, in \cref{sec:Kluyver}, we provide a generalization of Kluyver's formula, which also unifies various scattered results in the literature.

\section{Frobenius rings and their supercharacter theories} 

\subsection{Supercharacter theory and supercharacter table of a finite abelian group} \label{subsec:supercharacter_group}
We first recall the definition of a supercharacter theory on a finite abelian group $G.$ We refer the reader to \cite{brumbaugh2014supercharacters, diaconis2008supercharacters, fowler2014ramanujan} for further discussions on this topic. We remark that, in order to keep a consistent set of notations, our discussion here closely aligns with \cite[Section 2]{fowler2014ramanujan}. 
\begin{definition}
A supercharacter theory on $G$ is a pair $(\mathcal{K}, \mathcal{X})$ where  $\mathcal{K}=\{K_1, K_2, \ldots, K_m\}$ be a partition of $G$ and $\mathcal{X} =\{X_1, X_2, \ldots, X_m \}$ a partition of the dual group $\widehat{G}=\Hom(G, \C^{\times})$ of characters of $G$ which satisfies the following conditions 
    \begin{enumerate}
        \item $\{0 \} \in \mathcal{K};$
        \item $|\mathcal{X}| = |\mathcal{K}|;$
        \item For each $X_i \in \mathcal{X}$, the character sum 
        \[ \sigma_i = \sum_{\chi \in X_i} \chi\]
        is constant on each $K \in \mathcal{K}$;
        \item In the applications below, we will only use supercharacter theories satisfying the following additional property: for each \(X \in \mathcal X\) and each superclass \(K_i \in \mathcal K\), the sum
\[
\sum_{k \in K_i} \chi(k)
\]
is independent of the choice of \(\chi \in X\). We denote this common value by \(\Omega_{K_i}(X)\), or simply by \(\Omega_i(X)\). As explained in \cite[Section 2]{supercharacters_nguyen}, this extra condition gives a better framework to do spectral theory for certain associated graphs. Additionally, since we only work with undirected simple graphs, we will assume that $K_i=-K_i$ for each $1 \leq i \leq m.$
            \end{enumerate}
\end{definition}

For a supercharacter theory $(\mathcal{K}, \mathcal{X})$, the corresponding supercharacter table is the $m \times m$ matrix $S = (\sigma_i(K_j))_{i,j=1}^m.$ More precisely, 
\begin{equation}
S = \begin{tabular}{c|cccc}
 & $K_1$ & $K_2$ & $\cdots$ & $K_m$ \\
\hline
$\sigma_1$ & $\sigma_1(K_1)$ & $\sigma_1(K_2)$ & $\cdots$ & $\sigma_1(K_m)$ \\
$\sigma_2$ & $\sigma_2(K_1)$ & $\sigma_2(K_2)$ & $\cdots$ & $\sigma_2(K_m)$ \\
$\vdots$ & $\vdots$ & $\vdots$ & $\ddots$ & $\vdots$ \\
$\sigma_m$ & $\sigma_m(K_1)$ & $\sigma_m(K_2)$ & $\cdots$ & $\sigma_m(K_m)$ \\
\end{tabular}
\end{equation}

As explained in \cite{fowler2014ramanujan}, the matrix $S$ satisfies several orthogonality properties. To describe these properties, we  recall some concepts. First,  the space of complex-valued functions $f\colon G \to \C$ is equipped with the following natural inner product. 
\begin{equation} \label{eq:inner_product}
\langle f_1, f_2 \rangle = \frac{1}{|G|} \sum_{g \in G} f_1(g) \overline{f_2(g)}.
\end{equation}
A function $f\colon G \to \C$ is called a superclass function if $f$ is constant on each superclass in $\{K_1, K_2, \ldots, K_m \}.$ For a superclass function $f$, we will denote by $f(K_i)$ the value of $f$ at an element $x \in K_i.$ The space of all superclass functions with respect to the pair $(\mathcal{K}, \mathcal{X})$ will be denoted by $\mathcal{S}.$ This space $\mathcal{S}$ inherits the inner product structure from \cref{eq:inner_product}. For any two functions $f_1, f_2 \in \mathcal{S}$, their inner product can be expressed more concisely as:

\begin{equation} \label{eq:inner_product_new}
\langle f_1, f_2 \rangle = \frac{1}{|G|} \sum_{\ell=1}^m |K_\ell| f_1(K_\ell) \overline{f_2(K_\ell)}.
\end{equation}

As explained in \cite{fowler2014ramanujan}, $\{\sigma_i\}_{i=1}^m$ forms an orthogonal basis for $\mathcal{S}$. More precisely, we have 
\begin{equation} \label{eq:orthogonal}
\langle \sigma_i, \sigma_j \rangle = |X_i| \delta_{i,j}, 
\end{equation}
where $\delta$ is the Kronecker delta function.  By \cref{eq:inner_product_new} and \cref{eq:orthogonal} we have the following formula (see \cite[Equation 2.4]{fowler2014ramanujan}) 
\begin{equation} \label{eq:orthogonal_2}
\frac{1}{|G|} \sum_{\ell = 1}^m |K_\ell| \sigma_i(K_\ell) \overline{\sigma_j(K_\ell)}= |X_i| \delta_{i,j}. 
\end{equation}
Let 
\[ D = \diag\left(\sqrt{|K_1|}, \ldots, \sqrt{|K_m|} \right),\]
\[ L = \frac{1}{\sqrt{|G|}} \diag \left(\frac{1}{\sqrt{|X_1|}}, \ldots, \frac{1}{\sqrt{|X_m|}} \right). \]

Then by \cref{eq:orthogonal_2}, we have 
\[ (SD)(SD)^{*} = |G| \diag \left(|X_1|, \ldots, |X_m| \right). \]
Furthermore, if we let $U = LSD$ then

\[ U = \frac{1}{\sqrt{|G|}} \left[ \frac{\sigma_i(K_j) \sqrt{|K_j|}}{\sqrt{|X_i|}} \right]_{i,j=1}^m\]
and $U U^{*}=I$; namely $U$ is a unitary matrix. Since $U^{*} U = I$, we obtain the following orthogonality condition. 
\begin{equation} \label{eq:orthogonal_3}
    \frac{\sqrt{|K_i||K_j|}}{|G|} \sum_{\ell=1}^m \frac{\sigma_{\ell}(K_i) \overline{\sigma_{\ell}(K_j)}}{|X_\ell|} = \delta_{i,j}. 
\end{equation}

We state here a simple corollary of the fact that $U$ is unitary, which we will use later on. 
\begin{thm} \label{thm:det-general}
We have the following equality  
    \[ |\det(S)|^2 = |G|^{m} \prod_{i=1}^{m} \frac{|X_i|}{|K_i|} .\] 
    In particular, if $|K_i|=|X_i|$ for each $1 \leq i \leq m$, then $|\det(S)|=|G|^{\frac{m}{2}}.$
\end{thm}

We remark that by \cite[Proposition 2.3]{supercharacters_nguyen}, we have 
    \[ \frac{\Omega_j(X_i)}{|K_j|}  = \frac{\sigma_i(K_j)}{|X_i|}.\]
As a result, each orthogonality relation for $\sigma_i(K_j)$ can be converted to an equivalent one for $\Omega_j(X_i)$ and vice versa. 
For example, we can rewrite $U$ as 
 \[ U = \frac{1}{\sqrt{|G|}} \left[ \frac{\Omega_j(X_i) \sqrt{|X_i|}}{\sqrt{|K_j|}} \right]_{i,j=1}^m, \]
 and \cref{eq:orthogonal_3} is equivalent to 
 \begin{equation} \label{eq:orthogonal_4}
  \frac{1}{|G|\sqrt{|K_i| |K_j|}} \sum_{\ell=1}^m |X_\ell| \Omega_i(X_\ell) \overline{\Omega_j(X_{\ell})} = \delta_{i,j}. 
 \end{equation}

 We remark that since we assume $K_i = -K_i$, $\Omega_j(X_\ell) \in \mathbb{R}$ (see \cite[Proposition 2.3]{supercharacters_nguyen}). As a result, \cref{eq:orthogonal_4} is equivalent to 
 \begin{equation} \label{eq:orthogonal_5}
  \frac{1}{|G|\sqrt{|K_i| |K_j|}} \sum_{\ell=1}^m |X_\ell| \Omega_i(X_\ell) \Omega_j(X_{\ell}) = \delta_{i,j}. 
 \end{equation}
  In particular, when $i=j$, we have 

  \begin{equation} \label{eq:orthogonal_6}
  \frac{1}{|K_i||G|} \sum_{\ell=1}^m |X_\ell| \Omega_i(X_\ell)^2  = 1. 
 \end{equation}

\subsection{Orthogonality relations for Ramanujan sums over a finite Frobenius ring} \label{subsec:supercharacter_ring}
In this section, we apply the results to the case where the abelian group $G$ is the additive structure of a finite ring $R$. Here, we exploit the fact that a ring has another structure; namely the multiplicative structure. In general, it is unclear how to construct a supercharacter theory for a finite commutative ring. However, as explained in \cite{supercharacters_nguyen}, there is a class of rings for which such a theory naturally exists; namely the class of finite commutative Frobenius rings. We first recall this concept. 

\begin{definition} \label{defn:frobenius}
Let $R$ be a finite commutative ring. We say that $R$ is a Frobenius ring if $R$ is a $\Z/n$-algebra equipped with a non-degenerate $\Z/n$ linear functional $\psi: R \to \Z/n.$ Here, non-degenerate means that the kernel of $\psi$ does not contain any non-zero ideal in $R.$ 
\end{definition}

By \cite{honold2001characterization, lamprecht1953allgemeine}, there are some other equivalent characterizations of a finite Frobenius ring. For example, a local ring is Frobenius if and only if its socle module is cyclic. In other words, there exists an element $e \in R \setminus \{0 \}$ such that $Re$ is contained in all non-zero ideals in $R.$ A finite commutative ring is Frobenius if and only if it is a product of finite commutative local Frobenius rings.

Let $\zeta_n:= e^{\frac{2 \pi \bm{i}}{n}}$ be a fixed primitive $n$-th root of unity and $\chi: R \to \C^{\times}$ be the character defined by $\chi(a) =\zeta_n^{\psi(a)}.$ By \cite[Proposition 2.4]{nguyen2024integral}, the dual group $\Hom(R, \C^{\times})$ is a cyclic $R$-module generated by $\chi$; namely every character of $R$ is of the form $\chi_r$ where $\chi_r(a)=\chi(ra).$  We note that by the definition of $\chi$, the following identity holds for all $x, y \in R.$
\[ \chi_x(y) = \chi_y(x) = \chi_{xy}(1) =\chi(xy).\]

We now recall the definition of the Ramanujan sum $c(g,R)$.

\begin{definition} \label{def:ramanujan_sums}
Let $g \in R$.  The generalized Ramanujan sum $c(g, R)$ is defined as follows
\[ c(g,R) = c_{\psi}(g, R) = \sum_{a \in R^{\times}} \chi_g(a) = \sum_{a \in R^{\times}} \chi(ga). \]
\end{definition}

We have two remarks about this definition. First, $c_n(m)$ is precisely $c(m,\Z/n).$ Therefore, $c(g,R)$ is a natural generalization of the classical Ramanujan sum defined in \cref{eq:classical_ramanujan}. Second, at first glance, it seems that $c_{\psi}(g,R)$ depends on $\psi.$ However, as explained in \cite[Theorem 4.14]{nguyen2025gcd}, $c_{\psi}(g,R)$ is independent of $\psi.$ In fact, we have 
\begin{equation} \label{eq:c_g_R_formula}
    c(g, R) = \frac{\varphi(R)}{\varphi(R/\Ann_{R}(g))} c(1, R/\Ann_R(g)) = \frac{\varphi(R)}{\varphi(R/\Ann_{R}(g))} \mu(R/\Ann_R(g)). 
    \end{equation}
Here $\Ann_{R}(g)$ is the annihilator ideal of $g$; namely 
\[ \Ann_{R}(g) = \{r \in R \mid gr = 0\}.\]
Additionally, $\varphi$ is the generalized Euler function; $\varphi(R)=|R^{\times}|.$ Finally, $\mu$ is the generalized M\"obius function defined as follows. We recall that by the structure theorem for commutative Artinian rings, $R$ is isomorphic to a finite product of local rings $R \cong \prod_{i=1}^d R_i$.  Then, $\mu(R)$ is defined by the following rule. 

\[
\mu(R) =
\begin{cases} 
    1, & \text{if } |R|=1, \\ 
    0, & \text{if there exists } 1 \leq i \leq d \text{ such that } R_i \text{ is not a field,} \\
    (-1)^d, & \text{otherwise.}
\end{cases}
\]
Let $K_1, K_2, \ldots, K_m$ be the orbits of $R$ under the action of $R^{\times}.$ We will define $\tau(R)=m$ since in the case $R=\Z/n$, $\tau(R)=\tau(n)$--the number of positive divisors of $n.$ By definition, the elements in each $K_i$ are associates.  Without loss of generality, we will assume that $K_1 = R^{\times}.$  For each $1 \leq i \leq m$, let 
\[ X_i =\{\chi_g| g \in K_i \} .\] 
For convenience, we will denote by $K_g$ (respectively $X_g$) the class that contains $g$ (respectively $\chi_g).$ This is consistent with our convention that $K_1 = R^{\times}.$ Additionally, we will use the notations $\sigma_x(K_y)$ and $\Omega_x(K_y)$ for the appropriate sums. With these preparations, we are now able to recall the following proposition. 
\begin{prop}(See \cite[Theorem 4.1]{supercharacters_nguyen}) \label{prop:supercharacter_associated_units}
The pair $(\mathcal{K}, \mathcal{X})$ where $\mathcal{K} =\{K_1, K_2, \ldots, K_m \}$ and $\mathcal{X}=\{X_1, X_2, \ldots, X_m \}$ is a supercharacter theory for $(R,+).$ Furthermore,  for each $1 \leq i \leq m$
\[ |X_i|=|K_i| = \varphi(R/\Ann_{R}(g_i)),\]
where $g_i$ is an element in $K_i.$
\end{prop}

\begin{proof}
    The fact that $(\mathcal{K}, \mathcal{X})$ is a supercharacter theory is a direct consequence of \cite[Theorem 4.1]{supercharacters_nguyen}. Let us now prove the second part about the size of $|K_g|=|X_g|$. Let $\text{Stab}(g)$ be the stabilizer of $g$. We have  
\[ \text{Stab}(g) = \{u \in R^{\times}| ug =g \} = \{u \in R^{\times}| (u-1) \in \Ann_{R}(g) \} = \ker(R^{\times} \to (R/\Ann_{R}(g))^{\times}).  \]
By the orbit-stabilizer theorem, we have $|K_g| = \dfrac{\varphi(R)}{|\text{Stab(g)}|}$ and hence by the first isomorphism of groups
\[ |K_g| = |X_g| = \varphi(R/\Ann_{R}(g)). \qedhere \]
\end{proof}
We now explain the connection between the Ramanujan sums $c(g,R)$ and the values $\sigma_i(K_j)$ and $\Omega_i(K_j)$ described in \cref{subsec:supercharacter_group}.  By definition, 
\[ c(g,R) = \Omega_{K_1}(X_g) = \Omega_1(X_g).\]
We now use the orthogonality properties described in \cref{subsec:supercharacter_group} to derive several arithmetical properties of Ramanujan sums.  
We first have the following two theorems, which generalize a result of Carmichael in \cite{carmichael1932expansions} for classical Ramanujan sums. 
\begin{thm} \label{prop:first_orthogonal}
      \[ \sum_{g \in R} c(g,R)^2 = |R| \varphi(R),\] 
   where $\varphi(R) = |K_1|=|R^{\times}|.$
\end{thm}
\begin{proof}
    Using the facts that $|X_i|=|K_i|$, $c(g,R) = \Omega_1(X_g)$ and \cref{eq:orthogonal_6}, we have 
\begin{align*}
    \sum_{g \in R } c(g,R)^2 &= \sum_{\ell=1}^{m} \left[\sum_{g \in K_\ell} c(g,R)^2 \right]   \\ 
    &=\sum_{\ell=1}^m |K_\ell| \Omega_1(X_\ell)^2 =  \sum_{\ell=1}^m |X_\ell| \Omega_1(X_\ell)^2 = |K_1||R| = \varphi(R)|R|.
\end{align*}
\end{proof}
While the first orthogonality condition described in \cref{prop:first_orthogonal} follows directly from supercharacter theory, we can also prove it using a direct and elementary argument. 

\begin{proof}
    We have
    \begin{align*}
        \sum_{g\in R} c(g,R)^2 &= \sum_{g\in R}\left(\sum_{a\in R^{\times}} \chi_(ga)\right)\left(\sum_{b\in R^{\times}}\chi(gb)\right)\\
        &=\sum_{a\in R^{\times}}\sum_{b\in R^{\times}}\sum_{g\in R} \chi(g(a+b)) =\sum_{a\in R^{\times}}\sum_{b\in R^{\times}}\sum_{g\in R} \chi_{a+b}(g)
    \end{align*}
    The inner sum is 0 when $a+b\ne 0$ and it is $|R|$ when $a+b = 0$. Once $a$ is fixed, there is a unique $b$ such that $a+b =0$, therefore
  \[ \sum_{g \in R} c(g,R)^2 = \varphi(R) |R|. \]
\end{proof}
We describe another proof for \cref{prop:first_orthogonal} using graph theory. 
\begin{proof}
We recall that the unitary Cayley graph $G_R=\Gamma(R,R^{\times})$ is the graph with the following data 
\begin{enumerate}
    \item The vertex set of $G_R$ is $R.$
    \item Two vertices $a,b$ are adjacent if and only if $a-b \in R^{\times}.$
\end{enumerate}
As explained in \cite[Theorem 4.6]{nguyen2025gcd} (see also \cite[Theorem 4.6]{supercharacters_nguyen} for a generalization) the spectrum of $G_R$ is precisely $\{c(g,R)\}_{g \in R}.$ By the walk-counting formula, we have 
\[ \sum_{g \in R} c(g,R)^2 = \sum_{g \in R} \deg_{G_R}(g) = \varphi(R) |R|. \] 
\end{proof}

We discuss a slight generalization of \cref{prop:first_orthogonal}. First, we introduce the following definition. 
\begin{definition} \label{def:k-moment}
For each $k \geq 0$, we define $k$-th moment of Ramanujan sums as 
\[ M_k(R) = \sum_{g \in R} c(g,R)^{k}.\]
\end{definition}
By definition, we have $M_0(R)=|R|, M_1(R)=0, M_2(R) = \varphi(R)|R|$. 

\begin{rem}
    If $R_1$ and $R_2$ are two finite commutative Frobenius rings, then $M_{k}(R_1)= M_k(R_2)$ for all $k \geq 0$ if and only if $G_{R_1}$ and $G_{R_2}$ are cospectral. In particular, this happens if $G_{R_1}$ and $G_{R_2}$ are isomorphic. \cite[Theorem 5.3]{kiani2012unitary} and \cite[Theorem 4.1]{minac2024isomorphic} show that these two graphs are isomorphic if and only if $|R_1|=|R_2|$ and $R_1^{\s} \cong R_2^{\s}.$ Here $R^{\s} = R/\Rad(R)$ with $\Rad(R)$ being the Jacobson radical of $R.$ In \cite[Proposition 4.5]{minac2024isomorphic}, we show some examples of $R_1, R_2$ which are not isomorphic but their associated unitary Cayley graphs are. In fact, when these rings are a finite quotient of $\F_q[x]$, \cite[Proposition 4.5]{minac2024isomorphic} also  determines the number of isomorphism classes of unitary Cayley graphs. 
\end{rem}

\begin{prop} \label{prop:k-moment}
        Let $R= \prod_{i=1}^d R_i$ be a product of finite local Frobenius rings. For each $1 \leq i \leq d,$ let $f_i$ be the order of the residue field of $R_i.$ Then, for $k \geq 1$
\[ M_k(R) = |R|^k \prod_{i=1}^d \left[\left(1-\frac{1}{f_i} \right)^k + (-1)^k \frac{f_i-1}{f_i^k} \right] \] 
\end{prop}

\begin{proof}
We observe that $M_k(R)$ is multiplicative with respect to direct products; namely 
\[ M_k(R_1) M_k(R_2) =M_k(R_1 \times R_2).\] 
This follows directly from the property that for $r_1 \in R_1$ and $r_2 \in R_2 $
\[ c((r_1, r_2), R_1 \times R_2) = c(r_1, R_1) c(r_2, R_2).\]
Therefore, it is enough to prove this formula when $R$ is a local ring with the maximal ideal $\mathfrak{m}$ and $f=R/\mathfrak{m}$ is the order of the residue field. As explained after \cref{defn:frobenius}, there exists an element $e \in R \setminus \{0\}$ such that $Re$ is contained in all non-zero ideals in $R.$ If $g \in R$ such that $g \neq 0$ and $g$ is not associated to $e$ then $\Ann_{R}(g)$ is a proper sub-ideal of $\mathfrak{m}.$ Consequently $\mu(R/\Ann_R(g))=0$ and hence $c(g,R)=0.$ If $g=0$ then 
\[ c(0,R)=\varphi(R) = |R| \left(1-\frac{1}{f} \right) .\] 
If $g$ is associated with $e$ (there are exactly $f-1$ such elements) then by the short exact sequence 
\[ 1 \to 1 + \mathfrak{m} \to R^{\times} \to (R/\mathfrak{m})^{\times} \to 1, \]
we have 
\[ c(g,R) = c(e,R) = -\frac{\varphi(R)}{\varphi(R/\mathfrak{m})}= -|\mathfrak{m}|= \frac{-|R|}{f}\]
We conclude that 
\[ M_k(R) = |R|^k \left[\left(1-\frac{1}{f} \right)^k + (-1)^k \frac{f-1}{f^k} \right] \]
\end{proof}

We show that the graph-theoretic argument can be used to provide another proof for the derivation of the third moment $M_3(R)$ of Ramanujan sums. 
\begin{prop}
    Let $R= \prod_{i=1}^d R_i$ be a product of finite commutative local Frobenius rings. For each $1 \leq i \leq d,$ let $f_i$ be the order of the residue field of $R_i.$ Then 
    \[ M_3(R) = \sum_{g \in R} c(g,R)^3 = |R|^3 \prod_{i=1}^d \left(1-\frac{1}{f_i} \right) \left(1-\frac{2}{f_i} \right).\]
\end{prop}

\begin{proof}
    We know that $M_3(R)$ is precisely the number of closed walks of length $3$ in $G_R$. By \cite[Proposition 2.3]{unitary}, for two fixed vertices $a,b \in G_R$ which are adjacent, the number of vertices $c$ which are adjacent to both of them is 
    \[ |R| \prod_{i=1}^d \left(1-\frac{2}{f_i} \right).\]
    Therefore, the number of closed walks of length $3$ is 
    \[ M_3(R) = \varphi(R) |R| \times |R| \prod_{i=1}^d \left(1-\frac{2}{f_i} \right) =|R|^3 \prod_{i=1}^d \left(1-\frac{1}{f_i} \right) \left(1-\frac{2}{f_i} \right).  \]
\end{proof}

We now discuss the second orthogonality property. 
\begin{thm} \label{prop:second_orthogonal}
    Let $r_1, r_2 \in R$ such that $r_1$ and $r_2$ are not associates. Then 
    \[ \sum_{g \in R} c(r_1g, R) c(r_2g, R) =0.\]
\end{thm}

We first give a proof using supercharacter theory.

\begin{proof}
    For each \(r \in R\), we have
    \[
        c(gr,R)
        =
        \sum_{a \in R^\times} \chi_{gr}(a)
        =
        \sum_{a \in R^\times} \chi_g(ra)
        =
        |\operatorname{Stab}(r)|\,\Omega_r(X_g).
    \]
    Here \(\operatorname{Stab}(r)\) is the stabilizer of \(r\) under the action of \(R^\times\), and
    \(\Omega_r(X_g)\) denotes \(\Omega_{K_r}(X_g)\). Using this formula, \cref{eq:orthogonal_5},
    the fact that \(K_{r_1} \neq K_{r_2}\), and the equality \(|K_\ell|=|X_\ell|\), we obtain
    \begin{align*}
        \sum_{g \in R} c(r_1g,R)c(r_2g,R)
        &=
        \sum_{\ell=1}^{m}
        \left[
            \sum_{g \in K_\ell} c(r_1g,R)c(r_2g,R)
        \right] \\
        &=
        |\operatorname{Stab}(r_1)|\,|\operatorname{Stab}(r_2)|
        \sum_{\ell=1}^m
        |X_\ell|\Omega_{r_1}(X_\ell)\Omega_{r_2}(X_\ell) \\
        &=0.
    \end{align*}
\end{proof}

Similar to \cref{prop:first_orthogonal}, we can also give a more direct proof of \cref{prop:second_orthogonal}. 

\begin{proof}
   \begin{align*}
        \sum_{g\in R} c(r_1g,R)c(r_2g,R) &= \sum_{g\in R}\left(\sum_{a\in R^{\times}}\chi_{gr_1}(a)\right)\left(\sum_{b\in R^{\times}}\chi_{r_2g}(b) \right)\\
        &=\sum_{a\in R^{\times}}\sum_{b\in R^{\times}}\sum_{g\in R} \chi(g(r_1a+r_2b))= \sum_{a\in R^{\times}}\sum_{b\in R^{\times}}\sum_{g\in R} \chi_{r_1a+r_2b}(g).
    \end{align*} 
The inner sum is 0 when $r_1a+r_2b \ne 0$, but because $Rr_1 \ne Rr_2$, then $r_1a+r_2b \ne 0$ for any $a,b\in R^{\times}$ 
(see \cite[Lemma 2.1]{kaplansky1949elementary} \label{lem:kaplansky}).
The proof follows.
\end{proof}

We now use \cref{prop:first_orthogonal} and \cref{prop:second_orthogonal} to explain some well-known orthogonality relations for classical Ramanujan sums. For $m,n \in \Z$, we note that the classical Ramanujan sum $c_n(m)$ is precisely $c(m,\Z/n).$

\begin{cor}
We have the following orthogonality relations 
\begin{enumerate}
    \item Suppose that $n$ is a multiple of $k.$ Then 
\begin{equation*}
\sum_{m=1}^{n} c_{k}(m)^2 = \varphi(k) n. 
\end{equation*}
\item Let $p,q$ be two distinct positive integers. Then 
\[ \sum_{m=1}^{pq} c_p(m) c_q(m)=0. \]
\end{enumerate}
\end{cor}
\begin{proof}
Let us prove the first statement. By definition, $c_k(m)$ only depends on $m$ modulo $k.$ Therefore 
    \[ \sum_{m=1}^n c_k(m)^2 = \frac{n}{k} \sum_{m=1}^k c_k(m)^2 = \frac{n}{k} \sum_{m \in \Z/k} c(m, \Z/k)^2.\]
    By \cref{prop:first_orthogonal} we know that 
    \[ \sum_{m \in \Z/k} c(m, \Z/k)^2 = k \varphi(k). \]
    This shows that 
    \[ \sum_{m=1}^n c_k(m)^2 = n \varphi(k) .\]    
We now prove the second orthogonality relation. We remark that 
\[ c_p(m) = \sum_{\substack{1 \leq j \leq p \\ \gcd(j,p)=1}} \zeta_p^{mj} = \frac{\varphi(p)}{\varphi(pq)} \sum_{\substack{1 \leq j \leq pq \\ \gcd(j,pq)=1}} \zeta_p^{mj} =  \frac{\varphi(p)}{\varphi(pq)} \sum_{\substack{1 \leq j \leq pq \\ \gcd(j,pq)=1}} \zeta_{pq}^{mqj} = \frac{\varphi(p)}{\varphi(pq)} c(mq, \Z/pq). \] 
Consequently 
\[ \sum_{m=1}^{pq} c_p(m) c_q(m)= \frac{\varphi(p)\varphi(q)}{\varphi(pq)^2} \sum_{m \in \Z/pq} c(mq, \Z/pq) c(mp, \Z/pq). \]
Since $p \neq q$, $p$ and $q$ are non-associate elements in $\Z/pq.$ As a result, the above sum is $0$ by \cref{prop:second_orthogonal}. 
    
\end{proof}

\section{Ramanujan determinant} \label{sec:determinant}

Let $R$ be a finite commutative Frobenius ring. We will again denote by $(\mathcal{K}, \mathcal{X})$ the supercharacter theory on $(R,+)$ associated with $R^{\times}$ as explained in \cref{prop:supercharacter_associated_units}. For each $1 \leq i \leq m$, fix a representative $x_i$ of $K_i.$ We first have the following theorem which gives an explicit description of the associated supercharacter table. 

\begin{thm} \label{thm:table-ramanujan}
    Let $S$ be the supercharacter table associated with the pair $(\mathcal{K}, \mathcal{X}).$ Then $S=C_R$ where 
    \[ C_R =  [c(x_j, R/\Ann_{R}(x_i))]_{1 \leq i, j \leq m} ; \]
and 
\begin{equation} \label{eq:ramanujan_def}
c(x_j, R/\Ann_{R}(x_i))= \frac{\varphi(R/\Ann_{R}(x_i))}{\varphi(R/\Ann_{R}(x_ix_j))} \mu(R/\Ann_R(x_i x_j)). 
\end{equation}
\end{thm}

\begin{proof}
    By definition we have 
\begin{align*}
S_{ij} &= \sigma_i(K_j) = \sum_{\chi \in X_i} \chi(x_j) = \sum_{x \in K_i} \chi_{x}(x_j) = \sum_{x \in K_i} \chi_{x_j}(x) \\ 
 &= \frac{|K_i|}{\varphi(R)} \sum_{u \in R^{\times}} \chi_{x_j}(ux_i) = \frac{\varphi(R/\Ann_{R}(x_i))}{\varphi(R)} \sum_{u \in R^{\times}} \chi_{x_jx_i}(u)\\
 &= \frac{\varphi(R/\Ann_{R}(x_i))}{\varphi(R)}  c(x_ix_j, R)  = \frac{\varphi(R/\Ann_{R}(x_i))}{\varphi(R)}  \frac{\varphi(R)}{\varphi(R/\Ann_R(x_ix_j))}\mu(R/\Ann_R(x_ix_j)) \\
 &= \frac{\varphi(R/\Ann_{R}(x_i))}{\varphi(R/\Ann_{R}(x_ix_j))} \mu(R/\Ann_R(x_i x_j)).
\end{align*}
\end{proof}

By \cref{thm:det-general}, we also have the following which generalizes \cite[Theorem 1]{schlage2021determinant} and \cite[Proposition 2.4]{minac2024isomorphic}. 
\begin{thm} \label{thm:determinant}
Let $R$ be a finite commutative Frobenius ring. Then 
$|\det(C_R)| = |\det(S)|= |R|^{\frac{\tau(R)}{2}}$. 
\end{thm}

\begin{proof}
By Theorem \ref{thm:table-ramanujan}, we have \(C_R=S\). Additionally, $|K_i|=|X_i|$ for each $1 \leq i \leq m:=\tau(R)$. Therefore, applying Theorem \ref{thm:det-general} to the
additive group \((R,+)\), we conclude that
\[
|\det(S)|=|R|^{\tau(R)/2}.
\]
\end{proof}

We remark that while the definition of $C_R$ depends on the theory of Ramanujan sums, the final formula does not. In other words, for a finite commutative ring, it makes perfect sense to define $C_R$ as 
\[C_R =\left[\frac{\varphi(R/\Ann_{R}(x_i))}{\varphi(R/\Ann_{R}(x_ix_j))} \mu(R/\Ann_R(x_i x_j))\right]_{1 \leq i,j \leq m} .\]
Here $\{x_1, \ldots, x_m\}$ is a complete set of representatives for the orbits $R^{\times} \backslash R$; namely, they are pairwise non-associate in $R.$ We have the following theorem, which gives a new criterion for a finite commutative ring to be Frobenius. 

\begin{thm}
Let \(R\) be a finite commutative ring. Then $\det(C_R)\neq 0$
if and only if \(R\) is Frobenius.
\end{thm}

\begin{proof}
    We know that $\varphi$ and $\mu$ are multiplicative with respect to direct product; namely if $R=R_1 \times R_2$ then 
    \[ \varphi(R) = \varphi(R_1) \varphi(R_2), \mu(R) = \mu(R_1) \mu(R_2).\]
    Therefore, up to an ordering, $C_{R} = C_{R_1} \otimes C_{R_2}.$ Consequently,  
    \begin{equation} \label{eq:product}
    \det(C_R) = \det(C_{R_1})^{\tau(R_2)} \det(C_{R_2})^{\tau(R_1)}.
    \end{equation}
    By the Artin-Wedderburn theorem and the definition of commutative Frobenius rings, a finite commutative ring is Frobenius if and only if it is a product of  finite  commutative local Frobenius rings. By \cref{eq:product}, it is enough to prove that this proposition is true for the case $R$ is local.  If $R$ is Frobenius, then \cref{thm:determinant} implies that $\det(C_R) \neq 0.$ Let us assume that $R$ is not Frobenius. By \cite[Theorem 1]{honold2001characterization}, $R$ has two distinct minimal ideals  $Re_1$ and  $Re_2.$ The minimality condition implies that $\Ann_{R}(e_1)=\Ann_{R}(e_2)=\mathfrak{m}$ where $\mathfrak{m}$ is the maximal ideal of $R.$ We claim that for each $r \in R$
    \[\frac{\varphi(R/\Ann_{R}(r))}{\varphi(R/\Ann_{R}(e_1r))} \mu(R/\Ann_R(e_1 r)) = \frac{\varphi(R/\Ann_{R}(r))}{\varphi(R/\Ann_{R}(e_2r))} \mu(R/\Ann_R(e_2 r)).\]
    This will, of course, imply that $C_R$ has two identical columns and therefore it is singular and $\det(C_R)=0.$ To prove this fact, we consider two cases. If $r \in \mathfrak{m}$, then $re_1=r e_2=0.$ Therefore, both numbers are equal to $\varphi(R/\Ann_{R}(r))$. On the other hand, if $r \in R \setminus \mathfrak{m} = R^{\times}$, then $\Ann_{R}(e_1r)=\Ann_{R}(e_2r)=\mathfrak{m}$ and therefore these numbers are both equal to $-\frac{\varphi(R/\Ann_{R}(r))}{\varphi(R/\mathfrak{m})}$.

\end{proof}
While we are able to calculate the exact value for the determinant of $C_{R}$, the following more general question seems natural. 

\begin{question}
What can we say about the rank of $C_R$? 

\end{question}
\section{Kluyver's formula}  \label{sec:Kluyver}
In this section, we discuss an equivalent definition of Ramanujan sums, often called the Kluyver formula in the literature. Various works have studied special cases of this formula. For example, \cite{zheng2018polynomial} investigates the case where $R$ is a quotient of the polynomial ring $\F_q[x]$, while \cite{zheng2023ramanujan} examines the case where $R$ is a finite quotient of a Dedekind domain. The fact that these finite rings are Frobenius is proved in \cite[Theorem 3.8, Theorem 3.9]{nguyen2024integral}. In this context, our theorem below provides a unified approach to Kluyver's formula.

\begin{thm}
Let $R$ be a finite commutative Frobenius ring and $g \in R$. Then 

    \[ c(g,R) = \sum_{Rg \subset I} N(I) \mu(R/\Ann_{R}(I)).\]

    Here $N(I)$ is the order of the quotient ring $R/I$ and $\Ann_{R}(I)$ is the annihilator ideal of $I.$
\end{thm}

\begin{proof}
We remark that both sides of the equation are multiplicative with respect to the direct product. 
Indeed, if \(R=R_1\times R_2\), then every ideal of \(R\) is of the form
\(I_1\times I_2\), and we have
\[
N(I_1\times I_2)=N(I_1)N(I_2),
\qquad
\operatorname{Ann}_R(I_1\times I_2)
=
\operatorname{Ann}_{R_1}(I_1)\times \operatorname{Ann}_{R_2}(I_2).
\]
Since \(\mu\) is multiplicative, the right-hand side is multiplicative. Similarly, the left-hand side is also multiplicative. As a result, we only need to prove this statement when $R$ is local.

Since $R$ is a finite commutative local Frobenius ring, it has a unique minimal ideal; namely $I_0=Re$ for some $e \in R.$ Furthermore, $\Ann_{R}(I_0)=\Ann_{R}(e)=\mathfrak{m}$ where $\mathfrak{m}$ is the maximal ideal of $R$. We note that since $R$ is Frobenius, it has an elegant duality property: for each ideal $I$ in $R$, $\Ann_{R}(\Ann_{R}(I))=I$ (see \cite{honold2001characterization}). Let us consider the right-hand side. By definition, $\mu(R/\Ann_{R}(I))=0$ unless $\Ann_{R}(I)=R$ (when $I=0$) or $\Ann_{R}(I)= \mathfrak{m}$ (when $I=Re$). We  consider a few cases. 
    
    \textbf{Case 1.} $g=0.$ In this case, we have $c(g,R)= \varphi(R).$  On the other hand, the right-hand side is equal to 
    \[ |R| - |R/Re| = |R| -\frac{|R|}{|Re|} = |R|-\frac{|R|}{|R/\mathfrak{m}|} = |R| -|\mathfrak{m}| = \varphi(R) = c(g,R). \]
    
    \textbf{Case 2.} $g$ is associated with $e$; namely $Rg=Re.$ In this case, by \cref{eq:c_g_R_formula},  $c(g,R)=  -\frac{\varphi(R)}{\varphi(R/\mathfrak{m})}$ which is equal to $-|m|$ by the short exact sequence 
    \[ 1 \to 1 + \mathfrak{m} \to R^{\times} \to (R/\mathfrak{m})^{\times} \to 1 .\] 
    On the other hand, the only non-zero terms on the right side occur at $I=Re.$ Therefore, the right-hand side is equal to $-|R/Re| = -|\mathfrak{m}|.$

    \textbf{Case 3.} $g \neq 0$ and $Rg \neq Re.$ In this case, both sides are equal to $0.$
\end{proof}

\section*{Acknowledgements}
The first-named author is grateful to Professor Torsten Sander for sharing his insights and providing constant encouragement, which have been instrumental in advancing this line of research to its current stage. He also wishes to thank Professor Stephan R. Garcia for his helpful and encouraging correspondence. 
The main results of this paper were essentially completed during the spring of 2025, when the first-named author was a visiting assistant professor at Lake Forest College. We did not use AI in deriving these results. We did, however, use AI to assist with proofreading the paper.

\providecommand{\bysame}{\leavevmode\hbox to3em{\hrulefill}\thinspace}
\providecommand{\MR}{\relax\ifhmode\unskip\space\fi MR }
\providecommand{\MRhref}[2]{%
  \href{http://www.ams.org/mathscinet-getitem?mr=#1}{#2}
}
\providecommand{\href}[2]{#2}

\end{document}